\documentclass[12pt]{amsart}
\usepackage{amsfonts,amssymb,amsmath,amsthm,amscd,graphicx,mathtools}
\DeclarePairedDelimiter\floor{\lfloor}{\rfloor}
\newtheorem{theorem}{Theorem}[section]
\newtheorem{lemma}[theorem]{Lemma}
\newtheorem{proposition}[theorem]{Proposition}

\theoremstyle{definition}

\theoremstyle{remark}
\newtheorem{remark}[theorem]{Remark}

\numberwithin{equation}{section}

\newcommand{\lcr}{\raisebox{-5pt}{\mbox{}\hspace{1pt}
                 \includegraphics{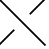}\hspace{1pt}\mbox{}}}
\newcommand{\ift}{\raisebox{-5pt}{\mbox{}\hspace{1pt}
                 \includegraphics{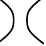}\hspace{1pt}\mbox{}}}
\newcommand{\zer}{\raisebox{-5pt}{\mbox{}\hspace{1pt}
                 \includegraphics{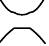}\hspace{1pt}\mbox{}}}
                 
\newcommand\scalemath[2]{\scalebox{#1}{\mbox{\ensuremath{\displaystyle #2}}}}

\title{Frobenius Algebras derived from the Kauffman bracket skein algebra}

\author{Nel Abdiel, Charles Frohman}
\begin{document}

\begin{abstract}

In this paper we study the skein modules of the surfaces, $\Sigma_{i,j}$ $(i,j)\in \{(0,2),(0,3),(1,0),(1,1)\}$ at $2N$-th roots of unity where $N\geq 3$ is an odd counting number and construct Frobenius algebras from them.\end{abstract}

\maketitle

\section{Introduction} A {\bf Frobenius algebra} $A$ over a field $k$ is an algebra equipped with a linear functional $\epsilon:A\rightarrow k$ that has no nontrivial principal  ideals in its kernel \cite{K}.  The Frobenius algebra is in addition {\bf symmetric} if for all $a,b\in A$, $\epsilon(ab)=\epsilon(ba)$. Frobenius algebras are central to the study of quantum field theories, and hence quantum invariants. Specifically, we are studying this phenomenon with an eye towards constructing a $4$-dimensional TQFT underlying the Kashaev invariant.
In this paper we explore symmetric Frobenius algebras derived from the Kauffman bracket skein algebra of some simple surfaces.

We explore two approaches to constructing algebras. The Kauffman bracket skein algebra $K_N(F)$ of a compact oriented surface where $A=e^{\pi {\bf i}/N}$, with $N\geq 3$ an odd integer, is a ring extension of the $SL_2\mathbb{C}$-characters $\chi(F)$ of the fundamental group of the surface. The first approach is to extend coefficients to the function field of the character variety to get an algebra $S^{-1}K_N(F)$ over the function field.  The second approach is given a {\bf place} $\phi:\chi(F)\rightarrow \mathbb{C}$, that is a homomorphism to the complex numbers, we specialize to get an algebra, $K_N(F)_{\phi}$, over the complex numbers.  In most cases, the derived algebra, equipped with the most obvious linear functional is a symmetric Frobenius algebra.
Frobenius algebras are finite dimensional. Identifying a finite generating set over the character ring is central to all our constructions.

The trace that makes the algebra into a symmetric Frobenius algebra comes from letting the algebra act on itself by left multiplication and then taking the normalized trace of the resulting linear map. The construction is agnostic with respect to the meaning of the trace.
A more conceptual approach to constructing the trace is being explored in \cite{FK}. The ethos of that approach is that the trace comes from understanding the skein module of a connected sum of copies of $S^1\times S^2$.

Throughout this paper,  the notation $\Sigma_{g,b}$ is  used to denote the compact oriented surface of genus $g$ with $b$ boundary components.  After a section broaching the background we study the construction applied to the Kauffman bracket skein algebras of $\Sigma_{0,2}$, $\Sigma_{0,3}$, $\Sigma_{1,0}$ and 
$\Sigma_{1,1}$.

\section{Background}

A framed link in a three-manifold $M$ is a disjoint union of oriented annuli in $M$.  The orientation of the annuli can be expressed by saying that each annulus has a preferred side.  Let $\mathcal{L}_M$ be the set of all equivalence classes of such annuli up to isotopy including the empty link.  Let $\mathbb{C}\mathcal{L}_M$ be the vector space with basis $\mathcal{L}_M$.  Choose an odd counting number $N$ and let $A=e^{\pi{\bf i}/N}$.  Finally $K_N(M)$ is the quotient of $\mathbb{C}\mathcal{L}_M$ by the submodule generated by the Kauffman bracket skein relations,  

\[\lcr-A\zer-A^{-1}\ift \]
and
\begin{equation*}\bigcirc \cup L+(A^2+A^{-2})L,\end{equation*} 
 where the framed links in 
each expression are
identical outside the balls pictured in the diagrams.  You should think of the
framed links in the ball as parallel to the arcs shown, with the preferred side up.

A quotient module is made up of cosets of the submodule you are modding out by.  The relators above are more familiarly written as relations between cosets.  In that form they look like,

\begin{equation}\displaystyle{\lcr=A\zer+A^{-1}\ift}\end{equation} 
and
\begin{equation}\bigcirc \cup L=-(A^2+A^{-2})L.\end{equation}

If $M=F\times [0,1]$ where $F$ is  an oriented surface,  then $K_N(M)$ is an algebra under stacking.  If $\alpha, \beta \in K_N(F\times [0,1])$ then $\alpha*\beta$ is the skein obtained by stacking $\alpha$ over $\beta$.  To emphasize that the algebra structure depends on how you realize $F\times [0,1]$ as a cylinder over the surface $F$, we denote it $K_N(F)$.  The planar surface with three boundary components $\Sigma_{0,3}$ and the surface of genus one with one boundary component $\Sigma_{1,1}$ are not homeomorphic, but $\Sigma_{0,3}\times [0,1]\cong \Sigma_{1,1}\times [0,1]$.
Meaning different skein algebras have the same underlying skein module.

A simple diagram $D$ on a surface $F$ is a system of disjoint simple closed curves so that none of the curves bounds a disk.  We include the empty diagram in this list.  By taking an annulus in $F$ about each simple closed curve 
in a diagram $D$ and orienting them to coincide with the orientation of $F$
we get a framed link corresponding to each simple diagram. Sometimes this is called the {\em blackboard framing}.
 No matter what  $A\in \mathbb{C}-\{0\}$ is used in the Kauffman bracket, the isotopy classes of framed links coming from isotopy classes of simple diagrams form a basis for
 the Kauffman bracket skein algebra, \cite{HP,BFK,SW}

The Tchebychev Polynomials of the first type  $T_k$ are defined recursively by
\begin{itemize}
\item $T_0(x) =2$, \item $T_1(x)= x$ and \item $T_{n+1}(x) = T_1(x) \cdot T_n(x) - T_{n-1}(x)$. \end{itemize}

 An easy inductive argument confirms that if $q\in \mathbb{C}-\{0\}$ then $T_k(q+q^{-1})=q^k+q^{-k}$.   This last formula can be taken as a characterization among polynomials of the $k$th Tchebychev polynomial of the first type.

\begin{lemma} If $p(z)$ is a polynomial so that for all $q\in \mathbb{C}-\{0\}$, $p(q+q^{-1})=q^k+q^{-k}$ then $p(z)=T_k(z)$. \end{lemma}

\proof If two analytic functions on $\mathbb{C}-\{0\}$ agree on a set with an accumulation point then they are equal. \qed

\begin{proposition}   \label{mult} 
For $m,n>0$, $T_m(T_n(x))=T_{mn}(x)$. Furthermore, for all $m,n\geq 0$,  $T_m(x)T_n(x)=T_{m+n}(x)+T_{|m-n|}(x)$.

\end{proposition}

\proof  The first formula follows from the observation  that $T_m(T_n)$ is a polynomial, and 
\[T_m(T_n(q+q^{-1}))=T_n(q^n+q^{-n})=q^{mn}+q^{-mn}.\] 
The second comes from seeing \[(q^m+q^{-m})(q^n+q^{-n})=(q^{m+n}+q^{-m-n})+(q^{|m-n|}+q^{-|m-n|}),\] and using the characterization of the Tchebychev polynomials of degree $k$ among polynomial functions.\qed

\vline

\begin{theorem}
(Bonahon-Wong \cite{C}). If  $M$ is a compact oriented three-manifold and $A= e^{\pi{\bf i}/N}$, there is a $\mathbb{C}$-linear map
\begin{center}
$\tau: K_{1}(M) \rightarrow K_{N}(M)$
\end{center}

\vline
\\
given by threading framed links with $T_N$. Any framed link in the image of $\tau$ is central in the sense that if $L' \cup K$ differs from $L \cup K$ by a crossing change of $L$ and $L'$ with $K$, then $T_N(L) \cup K = T_N(L') \cup K$. In the case that $M = F \times [0,1]$, the map

\vline

\begin{center}
$\tau: K_{1}(F) \rightarrow K_N(F)$
\end{center}
is an injective homomorphism of algebras.

\end{theorem}

It is easy to see that $K_1(M)$ is a ring under disjoint union. This is because at $A=-1$, the Kauffman bracket skein relation says that you can change crossings in a crossing ball of a framed link without changing the equivalence class of the induced skein. To take the product of two equivalence classes of framed links, choose representatives that are disjoint from one another and take their union. This can be extended to give a product on $K_1(M)$. Let $\sqrt{0}$ denote the nilradical of $K_1(M)$.  It is a theorem of Bullock, \cite{D}, that for any oriented compact $3$-manifold $K_1(M)/\sqrt{0}$ is canonically isomorphic to the coordinate ring of the $SL_2\mathbb{C}$-character variety of the fundamental group of $M$.  In the case that $M=F\times [0,1]$ the disjoint union product coincides with the stacking product. It is a theorem of Przytycki and Sikora \cite{PS} that in this case $\sqrt{0}=\{0\}$,  so $K_1(F)$ is the coordinate ring of the $SL_2\mathbb{C}$-character variety. To alleviate notational ambiguity, and to emphasize this fact, we always denote the image of the threading map by $\chi(M)$.

Next we need to broach some algebra. There are two methods by which we will derive an algebra over a field from the Kauffman bracket skein algebra. The first is by extending the base ring to its field of fractions, and the second is specializing at a place.

 Let $R\rightarrow A$ be a central ring extension, where is $R$ is an integral domain, $A$ is an associative ring with unit
and the inclusion of $R$ into $A$ is a ring homomorphism.  Since $R$ has no zero divisors, $S=R-\{0\}$ is multiplicatively closed.   We start with the set of ordered pairs $A\times S$, and we place an equivalence relation on $A\times S$ by $(a,s)$ is equivalent to $(b,t)$ if $at=bs$.  Denote the equivalence class of $(a,s)$ under this relation by $[a,s]$.  The set of equivalence classes is denoted $S^{-1}A$, and called the field of fractions of $S$. We denote the equivalence classes $[a,s]$ where $a\in R$ by $S^{-1}R$. We define multiplication of equivalence classes by $[a,s][b,t]=[ab,st]$ and addition by $[a,s]+[b,t]=[at+bs,st]$.  Under these operations $S^{-1}R$ is a field, and $S^{-1}A$ is an algebra over that field.

In the cases we are interested, the algebra is $K_N(F)$ and $R$ is $\chi(F)$.  This means that $S^{-1}R$ is the function field of the character variety of $\pi_1(F)$.

Next assume that $A$ is a finite extension of $R$. Thus, there exist $a_1,\ldots a_n$ so that every element of $A$ can be written as an $R$-linear combination of the $a_i$. The localization $S^{-1}A$  is a finite dimensional vector space over $S^{-1}R$.  In the passage to the field of quotients linear dependencies between the $a_i$ can be introduced, however some subset  $v_1,\ldots, v_k$ of the original generators is a basis.  Using this basis we get an algebra homomorphism,
\[ \rho:S^{-1}A\rightarrow M_{k,k}(S^{-1}R)\]
where $M_{k,k}(S^{-1}R)$ is the algebra of $k\times k$-matrices with coefficients in $S^{-1}R$.  The matrix $(a^j_i)$ associated to $\alpha\in S^{-1}A$ comes from writing  \[\alpha v_i=\sum_j\alpha^j_iv_j\]
for each $i$.    Finally let $Tr:S^{-1}A\rightarrow S^{-1}R$  be given by  $Tr(\alpha)=\frac{1}{k}\sum_ia^i_i$. We call the trace, scaled so that $Tr(1)=1$, the {\bf normalized trace}.
Notice that

\begin{proposition} \

\begin{itemize}
\item $Tr:S^{-1}A\rightarrow S^{-1}R$ is independent of the choice of basis.
\item $Tr(1)=1$.
\item $Tr$ is $S^{-1}R$ linear.
\item $Tr(\alpha\beta)=Tr(\beta\alpha)$.
\end{itemize}
\end{proposition}

\proof These all follow from elementary properties of the trace on $k\times k$-matrices. \qed

An $S^{-1}R$-linear map $Tr:S^{-1}A\rightarrow S^{-1}R$ is  {\em nondegenerate}, if for every $\alpha\neq 0 \in S^{-1}A$,  there is $\beta \in S^{-1}A$ with $Tr(\alpha\beta)\neq 0$. It is easy to show that this is equivalent to the statement that there are no nontrivial principal ideals contained in the kernel of $Tr$. If $Tr:S^{-1}A\rightarrow S^{-1}R$ is nondegenerate, then $S^{-1}A$ is a symmetric Frobenius algebra.  This is the path we will take in this paper towards proving that the localized Kauffman bracket skein algebra of $\Sigma_{i,j}$ where $(i,j)\in \{(0,2),(0,3),(1,0),(1,1)\}$ is a symmetric Frobenius algebra over the function field of the $SL_2\mathbb{C}$-character variety of the underlying surface.

The second construction of a Frobenius algebra is via specialization.
 In the cases we are interested in the threading map is an injection, and $\chi(M)$ is isomorphic to the coordinate ring of the $SL_2\mathbb{C}$-character variety of the fundamental group of $M$.  Two representations $\rho_1,\rho_2:G\rightarrow SL_2\mathbb{C}$ are {\bf trace equivalent} if for every $g\in G$, $tr(\rho_1(g))=tr(\rho_2(g))$, where $tr:SL_2\mathbb{C}\rightarrow \mathbb{C}$ is the trace.
 There is a one to one correspondence between ring homomorphisms $\phi:\chi(M)\rightarrow \mathbb{C}$ and trace equivalence classes of representations of the fundamental group of $M$ into $SL_2\mathbb{C}$.

Given $\phi:\chi(F)\rightarrow \mathbb{C}$ the set $I_{\phi}=(\ker{\phi })K_N(F)$ consisting of all sums where the terms are the product of a skein in $\ker{\phi}$ with an arbitrary element of $K_N(F)$, is a two sided ideal by dint of the fact that $\chi(F)$ is central, and $(\ker{\phi})\leq \chi(F)$ is an ideal. Hence we can form the quotient $K_N(F)_{\phi}=K_N(F)/I_{\phi}$ which is an algebra over the complex numbers. Once again, if $K_N(F)$ is finite dimensional over $\chi(F)$, the algebra $K_N(F)_{\phi}$ has a finite basis. Treating elements of $K_N(F)_{\phi}$ as operators by left multiplication we get 
\[ Tr_{\phi}:K_N(F)_{\phi}\rightarrow \mathbb{C},\] where the trace is normalized so that the identity map corresponds to $1$.  Most of the time $K_N(F)_{\phi}$ is a symmetric Frobenius algebra. This means there is a proper subvariety of the character variety of the fundamental group of $F$, so that if $\phi$ is not in that subvariety, then $K_N(F)_{\phi}$ is a symmetric Frobenius algebra.

\section{The Annulus }

The skein algebra of $K_N(\Sigma_{0,2})$ is naturally isomorphic to polynomials in one variable $\mathbb{C}[x]$. Under this isomorphism the variable  $x$ is the image of the core of the annulus with the blackboard framing.

Since the polynomials $T_k(x)$ have $x^k$ as their leading terms, we  can change bases, so that $K_N(\Sigma_{0,2})$ has basis $T_k(x)$ over the complex numbers, and multiplication obeys the product to sum formula $T_k(x)T_m(x)=T_{k+m}(x)+T_{|k-m|}(x)$.  

\begin{proposition}The image of the threading map, $\chi(\Sigma_{0,2})$, is the span of all $T_{Na}(x)$ where $a$ ranges over all non-negative integers.\end{proposition}

\proof This follows from Proposition \ref{mult}, since $\tau(T_k(x))=T_{Nk}(x)$.\qed

\begin{proposition} $K_N(\Sigma_{0,2})$ is a free module of rank $N$ over $\chi(\Sigma_{0,2})$ with basis $T_k(x)$ where $k$ ranges from $0$ to $N-1$.\end{proposition}

\proof

Suppose that $a$ is a non-negative integer and $0\leq b\leq N-1$. Solving the product to sum formula,  \[T_{aN+b}(x)=T_{aN}(x)T_b(x)-T_{|aN-b|}(x).\] If $a>0$ then  $aN>0$ and supposing that $b>0$, $aN-b<aN+b$. This means,
by induction on the largest $k$ so that $T_k(x)$ appears with nonzero coefficient in the skein, every element of $K_N(\Sigma_{0,2})$ can be written as a linear combination of $T_b(X)$ where the $b$ range from $0$ to $N-1$.   Therefore the proposed basis spans.

Suppose that \[\sum_{k=0}^{N-1}\chi_kT_k(x)=0,\] where the $\chi_k\in \chi(\Sigma_{1,0})$.
Rewrite each term using the basis $\{x^k\}$ for $K_N(\Sigma_{0,2})$. The leading term  of each summand as polynomials in $x$ must cancel. However the leading term of $\chi_kT_k(x)$ is of the form $\alpha_kx^{aN+k}$, where $\alpha_k$ is a complex number. The highest degree terms $\chi_kT_k(x)$ and $\chi_mT_m(x)$ where $m\neq k$ cannot cancel with each other.  Therefore all the leading terms of all $\chi_kT_k(x)$ are all zero. But this means all $\chi_k=0.$  Hence the $T_k(x)$ where $k$ ranges from $0$ to $N-1$ are linearly independent over $\chi(\Sigma_{0,2})$. \qed

If $\alpha\in K_N(\Sigma_{0,2})$ left multiplication by $\alpha$ defines a $\chi(\Sigma_{0,2})$-linear map,
\[L_{\alpha}:K_N(\Sigma_{0,2})\rightarrow K_N(\Sigma_{0,2}).\]

We can write $L_{\alpha}$ as an $N\times N$-matrix with coefficients in $\chi(\Sigma_{0,2})$.   For instance if $N=5$, and $\alpha=T_1(x)$ then,
the matrix is,

\[\begin{pmatrix} 0 & 1 & 0 & 0 & \frac{1}{2}T_5(x) \\ 2 & 0 & 1 & 0 & 0 \\ 0 & 1 & 0 & 1 & 0 \\ 0 & 0 & 1 & 0 & 1 \\ 0 & 0 & 0 & 1 & 0 \end{pmatrix}.\]

\begin{remark} Notice the determinant of this matrix is $T_5(x)$.  In general the determinant of the matrix of $T_k(x)$ acting on $K_N(\Sigma_{0,2})$ as a free module over $\chi(\Sigma_{0,2})$ is $T_{kN}(x)$.\end{remark}

More generally, using Proposition \ref{mult}, the matrix $Lk$ for left multiplication by $T_k(x)$, where $k$ ranges from $1$ to $N-1$ and the indices for the matrix run from $0$ to $N-1$, is given by
 \begin{equation} \label{left} Lk_{i,j}=\begin{cases} 2\delta_i^k & j=0 \\ \delta^{|k-j|}_i+\delta^{k+j}_i & (k+j\leq N-1) \wedge (j\neq 0)
      \\
     \delta^{|k-j|}_i-\delta^{2N-k-j}_i +T_N\delta^{k+j-N}_i& (k+j>N-1) \wedge (j\neq 0)\wedge ( k+j-N\neq 0 )\\
     \delta^{|k-j|}_i-\delta^{2N-k-j}_i +\frac{1}{2}T_N\delta^{0}_i& (k+j>N-1)  \wedge (j\neq 0)\wedge ( k+j-N= 0)
   \end{cases}
\end{equation}

Also $T_k(x)T_0(x)=2T_k(x)$ from the definition of $T_0$.

\begin{proposition} If $Tr:K_N(F)\rightarrow \chi(F)$ is $\frac{1}{N}$ times the trace of an element acting by left translation we find that $Tr(T_k)=0$ unless $N|k$. Furthermore if $N|k$ then $Tr(T_k)=T_k$. \end{proposition}

\proof     If $1\leq k \leq N-1$, There are two ways a diagonal entry can be nonzero,  if $|k-j|=j$ and $k+j<N$ or if $2N-k-j=j$ and $k+j>N-1$.  In fact, the $j\in \{0,1,\ldots, N-1\}$ satisfying the first condition can be placed in one-to-one correspondence with the $j\in\{0,\ldots,N-1\}$ that satisfy the second condition.  In the first case,
since $k>0$, $k-j=j$, but this implies that $k+N-j>N-1$ and $2N-k-(N-j)=N-j$, so $N-j$ satisfies the second condition.  Since $j$ and $N-j$ have the different parity they are not equal. From the formula for $Lk_{i,j}$, $j$ of the first kind gives rise to a $+1$ on the diagonal, and $j$ of the second kind gives rise to a $-1$ on the diagonal.  Hence all diagonal entries are $1,0$ or $-1$ and they sum to zero.  

If $N|k$ then $T_k\in \chi(\Sigma_{0,2})$ and $Lk$ is a diagonal matrix with each diagonal entry equal to $T_k$, meaning the normalized trace of $Lk$ is  $T_k$. \qed

\begin{remark}This gives a simple rule for computing the trace.  Write a skein in the form $\sum_{i}\alpha_iT_i(x)$ where the $\alpha_i\in S^{-1}\chi(\Sigma_{0,2})$, then $Tr(\alpha)=\sum_{N|i}\alpha_i T_i(x)$.  That is, throw out all the terms where $i$ is not divisible by $N$ and the sum of the remaining terms is the trace.

\end{remark}

We define  a pairing $\sigma:K_N(F)\otimes K_N(F)\rightarrow \chi(F)$ by 
$\sigma(\alpha,\beta)=Tr(\alpha\beta)$.

Unfortunately we cannot diagonalize the pairing unless we work over the field of fractions.  When $N=5$ the matrix of the pairing with respect to the basis $T_i$ where $i$ ranges from $0$ to $4$ is,
\begin{equation}\label{beta} \begin{pmatrix} 2T_0(x) & 0 & 0 & 0 & 0 \\ 0 & T_0(x) & 0 & 0 & T_5(x) \\ 0 & 0 & T_0(x) & T_5(x) & 0 \\ 0 & 0 & T_5(x) & T_0(x) & 0 \\ 0 & T_5(x) & 0 & 0 & T_0(x)\end{pmatrix}.\end{equation}

This example is general in the sense that the matrix of the pairing consists of a $1\times 1$ block with a $2$ and an $N-1\times N-1$ block that has $1$'s on the diagonal, $T_N(x)$ on the antidiagonal, and zeroes everywhere else.

\begin{proposition} The pairing is nondegenerate.  For every $\alpha\neq 0$ there exists $\beta\in K_N(\Sigma_{0,2})$ so that $Tr(\alpha\beta)\neq 0$. \end{proposition}

\proof  At level $N$ the pairing has $\beta(T_0(x)T_0(x))=2T_0(x)$, and $\beta(T_k(x)T_l(x))=T_0(x)$ if $k=l$ and $T_N(x)$ if $k+l=N$.  

The corresponding matrix has determinant \[2T_0(x)(T_0^2(x)-T_N(x)^2)^{\frac{N-1}{2}}.\] 
To see this, decompose the matrix of the pairing  into blocks. One block has determinant $2T_0(x)$.  The other block is easy to make upper triangular with row operations, so that the determinant is evident.
\qed 

\begin{theorem} $S^{-1}K_N(\Sigma_{0,2})$ is a Frobenius algebra over the function field of the character variety of $\pi_1(\Sigma_{0,2})$.\end{theorem}

\qed

\begin{proposition} $S^{-1}K_N(\Sigma_{0,2})$ is a field. \end{proposition}

\proof By definition  $S^{-1}\chi(\Sigma_{0,2})$ is a field.  We obtain  $S^{-1}K_N(\Sigma_{0,2})$ as an finite  extension of $K_1(\Sigma_{0,2})$ where we are adjoining $x$ that satisfies, 

\[ x^N=T_N(x)-\sum_{i=1}^{\floor{N/2}}(-1)^i\frac{N}{N-i}\binom{N-i}{i}x^{N-2i}.\]

Therefore $S^{-1}K_N(\Sigma_{0,2})$ is a finite extension of the field $S^{-1}\chi(\Sigma_{0,2})$, so it is itself a field. \qed

This gives a second proof that $K_N(\Sigma_{02})$ is Frobenius as it has no nontrivial ideals contained in $ker(Tr)$. The computational proof gives more information as knowing the determinant of the pairing tells you the locus of points where the specialized skein algebra is not Frobenius.

\begin{remark}More generally, if $F$ is a compact oriented surface and $S=\chi(F)-\{0\}$, then
in $S^{-1}K_N(F)$ any simple diagram is a unit.\end{remark}

Next we explore specializing $K_N(\Sigma_{0,2})$ at a place.  The character variety of the annulus can be understood as the complex plane. The places $\phi:\chi(\Sigma_{0,2})\rightarrow \mathbb{C}$ are determined by where $T_N(x)$ is sent.  Assume $\phi(T_N(x))=z$. Denote $K_N(\Sigma_{0,2})/ker(\phi)$ by $K_N(\Sigma_{0,2})_z$.  The action of $T_1$ by left multiplication comes from substituting $z$ for $T_N(x)$ in Equation \ref{left}.
In the case $N=5$ this is,

\[\begin{pmatrix} 0 & 1 & 0 & 0 & \frac{1}{2}z \\ 2 & 0 & 1 & 0 & 0 \\ 0 & 1 & 0 & 1 & 0 \\ 0 & 0 & 1 & 0 & 1 \\ 0 & 0 & 0 & 1 & 0 \end{pmatrix}.\]

The determinant of the matrix for left multiplication by $T_1$ specialized at $T_N(x)=z$ is $z$. You can see this by expanding the determinant by the last column. It is easy to see that only one term contributes to the final answer and it's contribution is $z$. Left multiplication by $T_1(x)$ is an invertible map unless $z=0$. If $z\neq 0$, then we can find $q\in \mathbb{C}$ so that $q^N
+q^{-N}=z$.  The roots of $T_N(x)=z$ are exactly the numbers $\zeta q+\zeta^{-1}q^{-1}$ where $\zeta$ is a $Nth$ root of unity.

\begin{theorem} $K_N(\Sigma_{0,2})_{z}$ is a Frobenius algebra over $\mathbb{C}$ so long as $z\neq \pm 2$. \end{theorem}.

\proof The pairing $\beta:K_N(\Sigma_{0,2})_{z}\otimes K_N(\Sigma_{0,2})_{z}\rightarrow  \mathbb{C}$ comes from the $N\times N$ matrix analogous to the one displayed in Equation \ref{beta}  by specializing. It's determinant is $2(T_0(x)^2-z^2)^{\frac{N-1}{2}}$. Recalling that $T_0(x)=2$, the pairing is nondegenerate so long as $z\neq \pm 2$. \qed

\section{The Pair of Pants}

The skein algebra of the pair of pants is isomorphic to $\mathbb{C}[x,y,z]$, the polynomials with complex coefficients in three variables. The three variables correspond to blackboard framed curves that are parallel to each of the boundary components.   Since $\mathbb{C}[x,y,z]$ is the tensor product of three copies of the polynomials in a single variable our analysis of $K_N(\Sigma_{0,2})$ can be used to analyze this case.   

In specific, the skeins $T_a(x)T_b(y)T_c(z)$ where $a,b,c\geq 0$ form a basis for $K_N(\Sigma_{0,3})$.  The image of the threading map $\chi(\Sigma_{0,3})$ is  $\{T_{Na}(x)T_{Nb}(y)T_{Nc}(z)\}$, where the $a,b$ and $c$ range over all non-negative integers. The skeins
$T_j(x)T_k(y)T_l(z)$, where $j,k,l$ range over $0\ldots N-1$ form a basis for $K_N(\Sigma_{0,3})$ as a module over $\chi(\Sigma_{0,3})$.  
With respect to this basis the matrix of left multiplication by  $T_j(x)T_k(y)T_l(z)$ is the tensor product of the three matrices coming from their actions on $K_N(\Sigma_{0,3})$.  Define $Tr:K_N(\Sigma_{0,3})\rightarrow \chi(\Sigma_{0,3})$ as $\frac{1}{N^3}$ times the trace of this matrix.

\begin{proposition} The map $Tr:K_N(\Sigma_{0,3})\rightarrow \chi(\Sigma_{0,3})$ is the identity when restricted to $\chi(\Sigma_{0,3})$
and if one of $j,k,l$ is not divisible by $N$ it sends $T_j(x)T_k(y)T_l(z)$ to zero. \end{proposition}
\proof The trace is just the tensor product of three copies of the trace on $K_N(\Sigma_{0,2})$. \qed

The evaluation of the trace is essentially the same as for the annulus. Write a skein in terms of the basis $T_i(x)T_j(y)T_k(z)$.  If one of the $i$, $j$ or $k$ is not divisible by $N$ throw it out, and the trace is the sum of the remaining terms.

We define the pairing $\sigma:K_N(\Sigma_{0,3})\otimes K_N(\Sigma_{0,3})\rightarrow \chi(\Sigma_{0,3})$ to be $\sigma(\alpha\otimes \beta)=Tr(\alpha\beta)$.

\begin{proposition} As an algebra over the field of fractions of $\chi(\Sigma_{0,3})$, $S^{-1}K_N(\Sigma_{0,3})$ is Frobenius. \end{proposition}

\proof The tensor product of Frobenius algebras is a Frobenius algebra. \qed

\begin{proposition} $S^{-1}K_N(\Sigma_{0,3})$ is a field. \end{proposition}

\proof The tensor product of fields is a field. \qed

\begin{proposition}The character variety of the free group on two generators is $\mathbb{C}^3$.
If $\phi:\mathbb{C}[x,y,z]\rightarrow \mathbb{C}$ doesn't correspond to evaluating the  variables $x$ $y$ or $z$  to $\pm 2$ the specialization of $\sigma$ is nondegernate and $K_{A}(\Sigma_{0,3})_{\phi}$ is a Frobenius algebra. \end{proposition}

\section{The Torus}

Simple closed curves on $\Sigma_{1,0}$ correspond to $(\pm p,\pm q)\in \mathbb{Z}\times \mathbb{Z}$ where $p$ and $q$ are relatively prime. The correspondence comes from expressing an oriented representative of the curve as an element of the first homology of the torus with respect to a standard basis. Since the curves are not oriented, $(p,q)=(-p,-q)$.

The simple diagrams form a basis for $K_N(\Sigma_{1,0})$.
Since two disjoint nontrivial simple closed curves on a torus are parallel, we can identify the simple diagrams with pairs $(\pm p,\pm q)\in \mathbb{Z}\times \mathbb{Z}$.  The ordered pair $(0,0)$ corresponds to the empty skein. The number of components of the simple diagram $(p,q)$ is the greatest common divisor of $(p,q)$ and the diagram consists of $d$ copies of the curve $(p/d/q/d)$.

Once again, the basis of simple diagrams  can be replaced by the basis of simple closed curves threaded with Tchebychev polynomials of the first type. For $(p,q)\neq (0,0)$ with $gcd(p,q) = 1$, we denote $(p,q)$, the $(p,q)$-curve on the torus. For $p,q$ with $gcd(p,q) = d$, we define 
\[(p,q)_T = T_d \big( \big( \frac{p}{d}, \frac{q}{d} \big) \big).\]
Finally, $(0,0)_T$ is twice the empty skein.
Since these skeins have exactly the simple diagrams as their ``lead'' term they form a basis for $K_N(\Sigma_{1,0})$. Multiplication in $K_N(\Sigma_{1,0})$ has a conveniently simple formula with respect to this basis.

\begin{theorem}
(Product to Sum Formula in $K_N(\Sigma_{1,0})$ \cite{B}). For any $p, q, r, s \in \mathbb{Z}$, one has 

\begin{center}
$(p,q)_T \ast (r,s)_T = A^{\begin{vmatrix} p& q \\ r& s \end{vmatrix}}(p+r,q+s)_T+A^{-\begin{vmatrix} p & q \\ r& s \end{vmatrix}}(p-r,q-s)_T.$
\end{center}

\end{theorem}

\begin{proposition} The image of the threading map $\chi(\Sigma_{1,0})$ is exactly the span of $(Np,Nq)_T$ where $p,q$ vary over $\mathbb{Z}$. It is exactly the center of $K_N(\Sigma_{1,0})$. \end{proposition}

\proof The characterization of the image of the threading map is the same for the annulus.  Using the product to sum formula you can show that if some $\alpha_{p,q}\neq 0$ and $p$ and $q$ are not both divisible by $N$ then
$\sum_{p,q}\alpha_{p,q}(p,q)\in K_N(\Sigma_{1,0})$ fails to commute with some $(r,s)$. \qed

\begin{theorem}\label{generators}
$K_N(\Sigma_{1,0})$ is finitely generated over $\chi(\Sigma_{1,0})$.
A spanning set $\mathcal{B}$ consists of $\{(a,b)_T\}$ so that $(a,b) \in (\{0\} \times \{0,1,..., N-1\}) \cup (\{1,..., N-1\} \times \{-\frac{N-1}{2},.., 0,..., N-1\}) \cup (\{N\} \times \{1,..., \frac{N-1}{2}\}) \cup (\{1,..., \frac{N-1}{2}\} \times \{N\})$.

\end{theorem}

The set $\mathcal{B}$ is a little complicated. Below is a diagram that pictorially represents the set when $N=5$. The dots correpond to points in $\mathbb{Z}\times \mathbb{Z}$ that belong to $\mathcal{B}$. The lowest point in the leftmost column is $(0,0)$.  

\vspace{.1in}

\begin{center}\includegraphics{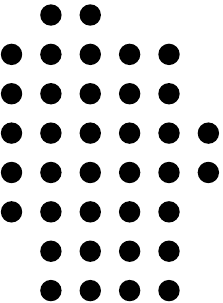}\end{center}

\proof  {\em Theorem \ref{generators}} The proof is in two steps. Let $\mathcal{A}=\{(p,q)_T\}$ where \[ (p,q)\in \{0\} \times\{0\} \cup \{0,N\} \times \{1,...N-1\} \cup \{1,..., N-1\} \times \{-N,...,0,..., N\}.\] First we show that $\mathcal{A}$ spans $K_N(\Sigma_{1,0})$. Next we show that $\mathcal{A}$ lies in the span of $\mathcal{B}$.

{\bf Step 1.} We want to show that as a module over $\chi(\Sigma_{1,0})$, $K_N(\Sigma_{1,0})$ is spanned by $(a,b)_T \in \mathcal{A}$. That is to say, any $(p,q)_T$ can be written as a $\chi(\Sigma_{1,0})$ linear combination of elements of $\mathcal{A}$.  
Let $(p,q)_T \in K_N(\Sigma_{1,0})$. For simplicity we work with pairs $(p,q)$ with $p\geq 0$, which we can do because $(p,q)=(-p,-q)$ as a simple diagram. Suppose $p \neq Nk$ for $k \in \mathbb{N}$ and write $p$ as $p=Np_1+a_1$ for some $0< a_1\leq N-1$. Assume that $p_1>0$. By the product to sum formula \[(p,q)_T = A^{-Np_1q}(Np_1,0)_T \ast (a_1,q)_T - A^{-2Np_1q}(Np_1-a_1,-q).\] Notice that $|Np_1-a_1| < |p|$, hence recursively applying this identity yields a linear combination of $(r,s)_T$ with coefficients in $\chi(\Sigma_{1,0})$ so that the $r$ lie in $\{0,1,\ldots N-1\}$.  To simplify the case $(p,q)_T$ where $p =Nk$ for some $k\geq 2$, use the product to sum formula with 
\[ (p,q)_T = (-1)^{-(k-1)q} (N(k-1), 0)_T \ast (N,q)_T- (-1)^{-2(k-1)q}(N(k-2),-q)_T\] to reduce to a linear combination of $(r,s)_T$ where $r\in \{0,N\}$. Thus $K_N(\Sigma_{1,0})$ is spanned over $\chi(\Sigma_{1,0})$ by $(r,s)_T$ with $r\in \{0,1,\ldots, N\}$.

Reduce the second entries in the same way, however now we must deal with negative integers. Suppose $q \neq Nm$ for $m \in \mathbb{Z}-\{0\}$ is non-negative and write $q$ modulo $N$, $q=Nq_1+b_1$ for some $0<b_1 \leq N-1$. If $q<0$ induct on $|q|$. As in the case of the first entry this process terminates in finitely many steps, leaving a finite linear combination of $(r,s)_T$ where $0\leq r \leq N$, and $-N\leq s \leq N$. In the case where $p = N$ reduce  $q$, to get a linear combination of $(r,s)_T$ so that $r=\{0,N\}$ and $1\leq s\leq N-1$.  

At this point it is clear that, \[\mathcal{A} = \{0\} \times\{0\} \cup \{0,N\} \times \{1,...N-1\} \cup \{1,..., N-1\} \times \{-N,...,0,..., N\}\] spans $K_N(\Sigma_{1,0})$. 

{\bf Step 2.} We finish the proof by showing $\mathcal{A}$ lies in the span of $\mathcal{B}$.
Notice $\mathcal{B}$ doesn't contain $(k, -N)$ for $1 \leq k \leq N-1$ but $\mathcal{A}$ does. By the product to sum formula these can be rewritten as elements in $\mathcal{B}$ in the following way \[(k,-N)_T = A^{kN}(0,N)_T \ast (k,0)_T - A^{-2kN}(k,N).\] 
This same argument gets rid of $(0,-k)$ and $(N,-k)$. To rewrite $(a,-b)_T \in  \{1,... N-1\} \times \{\frac{N-1}{2},..., N-1\}$ we use the following relation 

\begin{align*}
(a,-b)_T &= \frac{1}{2} ((-1)^{a+b+1}(N,N)_T \ast (N-a,b-N)_T)\\
&+ \frac{1}{2} ((-1)^{a}(0,N)_T \ast (a, N-b)_T + (-1)^{b}(N,0)_T \ast (N-a, b)_T).
\end{align*}
We can also rewrite $(a,N)_T$ for $a \in \{\frac{N+1}{2},..., N-1\}$ as 
\begin{align*}
(a,N)_T &= \frac{1}{2} ((-1)^{a+1}(N,N)_T \ast (N-a,0)_T) \\
&+ \frac{1}{2} ((-1)^{a}(0,N)_T \ast (a, 0)_T+ (N,0)_T \ast (N-a, N)_T).
 \end{align*} 
And the last case, $(N,b)_T$ for $b \in \{\frac{N+1}{2},..., N-1\}$ can be rewritten as 
\begin{align*}
(N,b)_T &= \frac{1}{2} ((-1)^{b+1}(N,N)_T \ast (0,N-b)_T )\\
&+ \frac{1}{2} ((0,N)_T \ast (N,N-b)_T + (-1)^{b}(N,0)_T \ast (0, b)_T).
 \end{align*}
 This proves that any element of $\mathcal{A}$ not in $\mathcal{B}$ can be obtained by a linear combination of elements of $\mathcal{B}$, hence $\mathcal{B}$ is the spanning set of $\mathcal{A}$ and thus of $K_N(\Sigma_{1,0})$.

\endproof

\begin{lemma}
Let $\mathcal{B^{\prime}}= \{(a,b)_T: 0 \leq a \leq N-1, 0\leq b\leq N-1\}$. The elements in $\mathcal{B}^{\prime}$ are linearly independent.
\end{lemma}

\begin{proof}
  
 Suppose that $\sum_{(a,b)_T\in \mathcal{B}'}\chi_{a,b}(a,b)_T=0$, where $\chi_{a,b}\in \chi(\Sigma_{1,0})$.   Using the characterization of $\chi(\Sigma_{1,0})$ as the span of $\{(Np,Nq)_T\}$ where $p,q\in \mathbb{Z}$,  \[ \sum_{a,b} \bigg( \sum_{i} \alpha_{a,b,i}(Np_i,Nq_i)_T \bigg) (a,b)_T=0 \] where $a, b \in \{0,1,..., N-1\}$, only finitely many  $\alpha_{a,b,i}\in \mathbb{C}$ are nonzero,  and $p_i\geq 0$ for all $i$ and if $p_i=0$ then $q_i\geq 0$. Using the product to sum formula, \[ \sum_{a,b} \bigg( \sum_{i} (-1)^{p_ib+q_ia}\alpha_{a,b,i}((Np_i+a,Nq_i+b)_T +(Np_i-a,Nq_i-b)_T) \bigg)=0. \] Let \[\mathcal{I} = \{(p,q)_T|(p,q)_T=(Np_i\pm a,Nq_i\pm b) \ \mathrm{and} \ \alpha_{a,b,i}\neq 0\}.\]

Assume $\mathcal{I}\neq \emptyset$.
Choose $(p,q)_T\in \mathcal{I}$ so that $p$ is maximal, and among all pairs with $p$ maximal, $q$ is maximal. Notice $p\geq 0$.

{\bf Case 1.} Suppose $p=Np_i+a$ for some $(Np_i+a,Nq_i+b)_T\in \mathcal{I}$ and $a>0$. Note that  $p\neq Np_j-a'$ for some $(Np_j-a',Nq_j-b')_T$  because $p$ was maximal and $Np_j+a'\geq Np_j-a'$, hence $a'=a=0$. This means that the only elements of $\mathcal{I}$ with first coordinate $p$ are of the form $(Np_j+a',Nq_j+b')_T$. Among these the maximal second coordinate is of the form $Nq_i+b'$.  There is a homomorphism $r:\mathbb{Z}\times \mathbb{Z}\rightarrow \mathbb{Z}_N\times \mathbb{Z}_N$ obtained by taking residue classes modulo $N$. If $r(Np_i+a,Nq_i+b)=r(Np_j+a',Nq_j+b')$ where $a,b,a',b'\in \{0,\ldots,N-1\}$ then $a=a'$ and $b=b'$. From this we see that there is a unique $(p_iN+a,q_i+b)_T$ corresponding to the choice of $(p,q)_T$. Therefore $\alpha_{a,b,i}=0$. 

{\bf Case 2.} Suppose that $p=Np_i$,  let $q=Nq_i+b$ be the maximal second coordinate among the pairs $(Np_i,q)_T\in \mathcal{I}$. If $b=0$ we are done, there is no other way of producing another $(Np_i,Nq_i)_T$ in the sum above, so assume $b \neq 0$. Moreover suppose there is another $\alpha_{a',b',j}\neq 0$ so that $(p,q)_T$ was equal to $(Np_j+a',Nq_j+b')_T$ or $( Np_j-a',Nq_j-b')_T$.  Notice that if $(p,q)_T=(Np_j+a',Nq_j+b')_T$ then $a'=0$ and  $(0,b)_T=(0,b')_T$, contradicting our assumption that there was more than one. Hence $(p,q)_T=(Np_j-a',Nq_j-b')_T$. 
Once again $a'=0$. Since $Nq_j+b'>Nq_j-b'$ and $\alpha_{a',b',j}\neq 0$ we have contradicted the maximality of $q$. \qed

Hence the elements in $\mathcal{B}^{\prime}$ are linearly independent.

\end{proof}

The following equation relates a multiple of every element in $\mathcal{B}-\mathcal{B^{\prime}}$ to two elements in $\mathcal{B^{\prime}}$.

\begin{align}\label{quad}
 0 &= (-1)^p[(2N,N)_T - (0,N)_T] \ast (p,q)_T \\
 &+(-1)^q[-(N,2N)_T+(N,0)_T] \ast (N-p,N-q)_T \nonumber \\
 &+[-(2N,2N)_T+2] \ast (p,q-N)_T. \nonumber
\end{align}

\noindent For example if $(a,b)_T \in \mathcal{B}-\mathcal{B}^{\prime}$ with $b <0$, set $a=p$ and $b=q-N$ in the previous formula, if $(a,N)_T \in \mathcal{B}-\mathcal{B}^{\prime}$, set $a=p$ and $N=q$ and last if $(N,b)_T \in \mathcal{B}-\mathcal{B}^{\prime}$, set $N=N-p$ and $b=N-q$. 

Let $S = \chi(\Sigma_{1,0})-\{0\}$. Based on the construction presented earlier, the field of fractions $S^{-1}\chi(\Sigma_{1,0})$, is defined by $\chi(\Sigma_{1,0}) \times S$ with the following equivalence relation, $(z,s) \sim (z^{\prime}, s^{\prime})$ if $zs^{\prime}=z^{\prime} s$, where $z$ and $z^{\prime} \in \chi(\Sigma_{1,0})$ and $s$ and $s^{\prime} \in S$. Similarly we construct $S^{-1}K_N(\Sigma_{1,0})$.

Equation 5.1 under the image of $\tau$ then says that 
\begin{align*}
(p,q-N)_T &= \frac{(-1)^p((2N,N)_T - (0,N)_T) \ast (p,q)_T}{-(-(2N,2N)_T+2)} \\
&+\frac{(-1)^q(-(N,2N)_T+(N,0)_T) \ast (N-p,N-q)_T}{-(-(2N,2N)_T+2)} .
\end{align*} 
This relation is what makes the following proposition hold, expressing the image of elements of $\mathcal{B}-\mathcal{B^{\prime}}$ linear combinations of $\mathcal{B^{\prime}}$.

\begin{proposition}
$S^{-1}K_N(\Sigma_{1,0})$ is vector space of dimension $N^2$ over $S^{-1}\chi(\Sigma_{1,0})$.
\end{proposition}

\begin{proof}

Let $\mathcal{D}^{\prime} = \{[(a,b)_T,1]: 0 \leq a \leq N-1, 0\leq b\leq N-1\}$. Let $[L,s] \in S^{-1}K_N(\Sigma_{1,0})$. By theorem 5.3 $L \in K_A(\Sigma_{1,0})$ can be written as \[\sum_{\beta \in \mathcal{B}} \chi_{\beta}\beta\] where $\chi_{\beta} \in \chi(\Sigma_{1,0})$. Because of equation 5.1 and the construction of $S^{-1}K_N(\Sigma_{1,0})$, \[[L,s] = [\sum_{k=0}^{N^2} \chi_k (a_k,b_k)_T,s] =  \sum_{k=0}^{N^2} [\chi_k (a_k,b_k)_T,s] = \sum_{k=0}^{N^2} [\chi_k,s] \ast [(a_k,b_k)_T,1]\] where $(a_k,b_k)$ are the elements of $\mathcal{B}^{\prime}$ ordered lexicographically. This shows that $S^{-1}K_N(\Sigma_{1,0})$ is spanned by $\mathcal{D^{\prime}}$ as an $S^{-1}\chi(\Sigma_{1,0})$-module. We showed in lemma 5.4 that $B^{\prime}$ is linearly independent from which it follows that $\mathcal{D^{\prime}}$ is linearly independent proving the proposition.

\end{proof}

Now that we have the previous proposition we can state the following.

\begin{proposition}
$K_N(\Sigma_{1,0})$ is not a free module over $\chi(\Sigma_{1,0})$ and has rank greater than $N^2$.
\end{proposition}

\begin{proof}
This simply follows from the fact that the coefficients in the previous proposition are unique and the linear combinations that yield the elements of $\mathcal{B}-\mathcal{B}^{\prime}$ are not in the image of elements of $\mathcal{B}^{\prime}$. Hence the rank of $K_N(\Sigma_{1,0})$ over $\chi(\Sigma_{1,0})$ is greater than dimension of $\mathcal{B}^{\prime}$ and because of equation 5.1 it is not free.
\end{proof}

It will be helpful to have a collection of bases for $S^{-1}K_N(\Sigma_{1,0})$.

\begin{proposition} Let $\alpha,\beta$ be framed links in $\Sigma_{1,0}\times [0,1]$ coming from a pair of simple closed curves $a$ and $b$ on $\Sigma_{1,0}$ by using the blackboard framing, where $a$ and $b$ intersect one another transversely in a single point.  The set of skeins $\mathcal{C}=\{T_p(\alpha)T_q(\beta)\}$ where $p$ and $q$ range from $0$ to $N-1$, form a basis for $S^{-1}K_N(\Sigma_{1,0})$ as a vector space over $S^{-1}\chi(\Sigma_{1,0})$. \end{proposition}

\proof  There is a diffeomorphism $h$ of $\Sigma_{1,0}$ that takes $a$ to the $(1,0)$ curve and $b$ to the $(0,1)$ curve.  If $h$ is orientation preserving it extends to an orientation preserving diffeomorphism $h\times Id:\Sigma_{1,0}\times [0,1]\rightarrow \Sigma_{1,0}\times [0,1]$. Let $s:[0,1]\rightarrow [0,1]$ be the diffeomorphism $s(t)=1-t$. If $h$ is orientation reversing then it extends to an orientation preserving diffeomorphism $h\times s:\Sigma_{1,0}\times [0,1]\rightarrow \Sigma_{1,0}\times [0,1]$.  An orientation preserving homeomorphism of a three-manifold $M$ induces an automorphism of $K_N(M)$. This means we only need to prove that $\{(p,0)_T*(0,q)_T\}$ and  $\{(0,q)_T*(p,0)_T\}$where $p,q\in \{0,\ldots, N-1\}$ form a basis of $S^{-1}K_N(\Sigma_{1,0})$.  The two proofs are similar so we restrict our attention to $\mathcal{C}=\{(p,0)_T*(0,q)_T\}$ . Since the set has $N^2$ elements, to prove that it is a basis we only need to prove that $\mathcal{C}$ is linearly independent over the function field of the character variety of $\Sigma_{1,0}$.

  As a set we realize the character variety of the fundamental group of the torus $X(\Sigma_{1,0})$ to be the set of {\em trace equivalence classes} of representations $\rho:\pi_1(\Sigma_{1,0})\rightarrow SL_2\mathbb{C}$.  Denote the trace equivalence class of $\rho$ by $[\rho]$. The coordinate ring 
$C[X(\Sigma_{1,0})]$ of $X(\Sigma_{1,0})$ is the ring generated by functions $\eta_{g}:X(\Sigma_{1,0})\rightarrow \mathbb{C}$, given by $\eta_g([\rho])=-tr(\rho(g))$.
If $g\in \pi_1(\Sigma_{1,0})$ is homotopic to a simple closed curve, then the isomorphism \cite{D},
\[\Phi:K_1(\Sigma_{1,0})\rightarrow C[X(\Sigma_{1,0})]\] sends the simple diagram corresponding to $g$ to $\eta_g$.

To facilitate the computation we will use a specific $2$-fold branched cover of the character variety of the fundamental group of the torus.  If $\lambda,\mu \in \mathbb{C}^*$, let $\rho_{\lambda,\mu}:\pi_1(\Sigma_{1,0})\rightarrow SL_2\mathbb{C}$
be the homomorphism with \[\rho_{\lambda,\mu}((1,0))=\begin{pmatrix} \lambda & 0 \\ 0 &\lambda^{-1} \end{pmatrix} \ \rho_{\lambda,\mu}((0,1))=\begin{pmatrix} \mu & 0 \\ 0 &\mu^{-1} \end{pmatrix} .\]  We then define
\[ \Theta:\mathbb{C}^* \times \mathbb{C}^* \rightarrow X(\Sigma_{1,0}),\]
by $\Theta(\lambda,\mu)=[\rho_{\lambda,\mu}]$.  The pullback $\Theta^*:C[X(\Sigma_{1,0})]\rightarrow C[\lambda^{\pm 1},\mu^{\pm 1}],$ is given by
\[ \Theta^*\circ \Phi((p,q)_T)=(-1)^d(\lambda^p\mu^q+\lambda^{-p}\mu^{-q}),\]
where $d$ is the greatest commond divisor of $p$ and $q$.

Partition $\mathcal{C}$ into two subsets. Let $\mathcal{C}'$ be the subset where $p=0$ or $q=0$, and let  $\mathcal{C}''$ be the subset where $p,q$ are both nonzero.  Partition  $\mathcal{C}''$ further, into sets of four basis elements $(a,b)_T$ where for some fixed $p,q\in \{1,\ldots,N-1\}$, $a\in \{p,N-p\}$ and $b \in \{q,N-Q\}$.
For instance when $N=3$, \[\mathcal{C}'=\{2(0,0)_T,2(1,0)_T,2(2,0)_T,2(0,1)_T,2(0,2)_T\} \] and $\mathcal{C}''$ consists of one set of four, \[\{(1,2)_T,(2,1)_T,(1,1)_T,(2,2)_T\}.\] In general $\mathcal{C}'$ has $2N-1$ elements, and $\mathcal{C}''$ decomposes into $\frac{(N-1)^2}{4}$ sets of four elements.  The spans of each one of these sets are independent from one another.  

Given $p,q\in \{1,\ldots N-1\}$. We need to show that \[\{(p,0)_T*(0,q)_T,(N-p,0)_T*(0,q)_T,(p,0)_T*(0,N-q)_T,(N-p,0)_T*(0,N-q)_T\}\] are linearly independent.
To this end we rewrite these in terms of the basis $\mathcal{B}$ using relation \ref{quad}.

The change of basis matrix, where the columns correspond in order to $(p,q)_T,(p,N-q)_T,(N-p,q)_T$ and $(N-p,N-q)_T$, and the rows correspond in order to $(p,0)_T*(0,q)_T,(p,0)_T*(0,N-q)_T,(N-p,0)_T*(0,q)_T$ and $(N-p,0)_T*(0,N-q)_T$,
is given by

\[ \scalemath{0.60}{\begin{pmatrix} A^{pq} & \frac{(-1)^{p+1}A^{-pq}((2N,N)_T-(0,N)_T)}{2-(2N,2N)_T} &\frac{(-1)^{q}A^{-pq}((N,2N)_T-(N,0)_T)}{2-(2N,2N)_T} & 0 \\ \frac{(-1)^{p+1}A^{-(N-p)q}((2N,N)_T-(0,N)_T)}{2-(2N,2N)_T} & A^{p(N-q)} & 0 &\frac{(-1)^{N-q}A^{-p(N-q)}((N,2N)_T-(N,0)_T)}{2-(2N,2N)_T} \\ \frac{(-1)^{q}A^{-(N-p)q}((N,2N)_T-(N,0)_T)}{2-(2N,2N)_T} & 0 & A^{p(N-q)q} &  \frac{(-1)^{N-p+1}A^{-(N-p)q}((2N,N)_T-(0,N)_T)}{2-(2N,2N)_T}  \\ 0 &\frac{(-1)^{N-q}A^{-(N-p)(N-q)}((N,2N)_T-(N,0)_T)}{2-(2N,2N)_T} & \frac{(-1)^{N-p+1}A^{-(N-p)(N-q)}((2N,N)_T-(0,N)_T)}{2-(2N,2N)_T} & A^{(N-p)(N-q)} \end{pmatrix} }.\]

To see this matrix is nonsingular, pull it back to  a matrix with coefficients from the field of quotients of $\mathbb{C}[\lambda^{\pm 1},\mu^{\pm 1}]$ by using $\Theta^* \circ \Phi$. Next multiply by $2\lambda^2\mu^2-\lambda^4\mu^4-1$ to clear the denominator so that the result is a matrix with coefficients in $\mathbb{C}[\lambda,\mu]$.  Simple inspection reveals that highest power of $\lambda$ and $\mu$ that appears is $4$, taking the coefficient of $\lambda^4\mu^4$ yields,
\[\begin{pmatrix} A^{pq} & 0 & 0 & 0 \\ 0 & A^{p(N-q)} & 0 & 0 \\ 0 & 0 & A^{(N-p)q} & 0 \\ 0 & 0 & 0 & A^{(N-p)(N-q)} \end{pmatrix},\] which has nontrivial determinant.

The change of basis matrix between the basis $\mathcal{B}^{\prime}$ and $\mathcal{C}$ decomposes into a block which is the $2N-1$ dimensional identity matrix times $2$ corresponding to $(p,0)_T*(0,0)_T$ and $(0,0)_T*(0,q)_T$ , and $(\frac{(N-1)^2}{2})^2$ blocks coming from changing basis between $\{(p,q)_T,(p,N-q)_T,(N-p,q)_T,(N-p,N-q)_t\}$ and $\{(p,0)_T*(0,q)_T,(p,0)_T*(0,N-q)_T,(N-p,0)_T*(0,q)_T,(N-p,0)_T*(0,N-q)_T\}$ where $p$ and $q$ range over $1,\ldots, \frac{N-1}{2}$. The above computation shows that each one of these blocks is nonsingular. Therefore $\mathcal{C}$ is a basis.  \qed 

The next step is to compute the trace.  By this we mean let $S^{-1}K_N(\Sigma_{1,0})$ act on itself on the left, and take $\frac{1}{N^2}$ times the trace of the resulting matrix.

\begin{theorem} The trace of $\sum_{p,q}\alpha_{p,q}(p,q)_T$ where $\alpha_{p,q}\in S^{-1}\chi(\Sigma_{1,0})$ is \[\sum_{p,q \ \mathrm{where} \ N|p \ \mathrm{and} N|q}\alpha_{p,q}(p,q)_T.\]

\end{theorem}

\proof The linearity of the trace means we only need to compute the trace of $(p,q)_T$ and see that it is zero unless $N|p$ and $N|q$, in which case it is $(p,q)_T$.  Since orientation preserving homeomorphisms of $\Sigma_{1,0}$ induce automorphisms of $S^{-1}\chi(\Sigma_{1,0})$ and $S^{-1}K_N(\Sigma_{1,0})$ that are coherent with one another, and there is an orientation homeomorphism taking and simple closed curve to the $(1,0)$ curve, we only need to compute the trace of $(k,0)$.

It doesn't make any difference what basis we choose for $K_N(\Sigma_{1,0})$ so we choose $\{(p,0)_T*(0,q)_T\}$ where $p,q\in \{0,\ldots,N-1\}$..  From the product to sum formula,
\[ (k,0)_T*(p,0)_T*(0,q)_T=(p+k,0)_T*(0,q)_T+(p-k,0)_T*(0,q)_T.\]
Computations similar to the ones for computing the trace on the annulus complete the proof. \qed

Here is a construction that will come in handy. Suppose that $C$ is a simple closed curve on the surface $F$ and $A_0$ and $A_1$ are annular neighborhoods of $C\times\{0\}$ and $C\times \{1\}$ in $F\times \{0,1\}$.  By gluing an annulus cross an interval into $A_i$ we can make $K_N(F)$ into a left-module ( gluing into $A_1$) or a right-module over $K_N(\Sigma_{0,2})$.  Since the inclusion map sends $\chi(\Sigma_{0,2})$ into $\chi(F)$ the each action  extends to make $S^{-1}K_N(F)$ into a vector space over $S^{-1}K_N(\Sigma_{0,2})$.

For instance $S^{-1}K_N(\Sigma_{1,0})$ is a vector space over $S^{-1}K_N(\Sigma_{0,2})$ where the action comes from gluing the annulus into an annulus in $\Sigma_{1,0}\times \{1\}$ corresponding to the $(1,0)$ curve.
The basis $(p,0)_T*(0,q)_T$ where $p,q\in \{0,1,\ldots,N-1\}$ corresponds to  a splitting of $K_N(\Sigma_{1,0})$ as a direct sum of vector spaces over $S^{-1}(\Sigma_{0,2})$ spanned by $\chi(\Sigma_{1,0})(0,q)_T$.
\[ S^{-1}K_N(\Sigma_{1,0})=\bigoplus_q (S^{-1}K_N(\Sigma_{0,2}))\chi(\Sigma_{1,0})(0,q)_T.\]
If $T_k(x)\in S^{-1}K_N(\Sigma_{0,2})$ it acts as $(k,0)_T$ on the left.

A similar construction, coming from gluing the annulus into an annulus in $\Sigma_{1,0}$ coming from the $(0,1)$ curve allows us to write $K_N(\Sigma_{1,0})$ as a direct sum, over $S^{-1}K_N(\Sigma_{0,2})$ acting on the right,
\[ S^{-1}K_N(\Sigma_{1,0})=\bigoplus_p \chi(\Sigma_{1,0})(p,0)_T(S^{-1}K_N(\Sigma_{0,2})).\]
Here $T_k(x)\in S^{-1}K_N(\Sigma_{0,2})$ it acts as $(0,k)_T$ on the left.

\begin{theorem} $S^{-1}K_N(\Sigma_{1,0})$ is a Frobenius algebra over $S^{-1}\chi(\Sigma_{1,0})$ with the normalized trace $Tr:S^{-1}K_N(\Sigma_{1,0})\rightarrow S^{-1}\chi(\Sigma_{1,0})$. \end{theorem}

\proof  We prove that there are no nontrivial principal ideals in the kernel of $Tr$.
Suppose that $I=(\sum_{p,q}\alpha_{p,q}(p,0)_T*(0,q)_T)$ is a principal ideal where $\alpha_{p,q}\in S^{-1}\chi(\Sigma_{1,0})$, $p,q\in \{0,\ldots,N-1\}$ and some $\alpha_{a,b}\neq 0$.  Since $S^{-1}K_N(\Sigma_{0,2})$ is a field there exist $T_a(x)^{-1}$ and $T_b(x)^{-1}$.  Using the actions above to include these into $\Sigma_{1,0}\times [0,1]$ there are skeins $(a,0)_T^{-1}$ and $(0,b)_T^{-1}$ in $S^{-1}K_N(\Sigma_{1,0})$.  That is, the skeins $(p,q)_T$ are units.
Multiplying the generator of the principle ideal on the left by $(a,0)_T^{-1}$ and on the right by $(0,b)_T^{-1}$ we get a skein of the form,
\[ (0,0)_T+\sum_{(p,q)\neq (0,0)}\beta_{p,q}(p,0)_T*(0,q)_T\]
with $\beta_{p,q}\in S^{-1}\chi(\Sigma_{1,0})$ is in $I$.  However 
\[Tr((0,0)_T+\sum_{(p,q)\neq (0,0)}\beta_{p,q}(p,0)_T*(0,q)_T)=(0,0)_T\neq 0.\]

\qed

Next we compute the determinant of  the pairing $<\alpha,\beta>=Tr(\alpha\beta)$.  From the product to sum formula and the linearity of trace,
\[ Tr((p,q)_T \ast (r,s)_T)=A^{\left|\begin{matrix} p & q \\ r & s \end{matrix}\right|}Tr((p+r,q+s)_T)+A^{-\left|\begin{matrix} p & q \\ r & s \end{matrix}\right|}Tr((p-r,q-s)_T).\]

This means that $<(p,q)_T,(r,s)_T>=0$ unless $N$ divides both $p+r$ and $q+s$, or $N$ divides both $p-r$ and $q-s$. As $N$ is odd, both case don't occur at the same time unless one of $(p,q)_T$ or $(r,s)_T$ is $(0,0)_T$.  In the first case, $<(p,q)_T,(r,s)_T>=(-1)^{ps+qr}(p+r,q+s)_T$ and in the second case $<(p,q)_T,(r,s)_T>=(-1)^{ps+qr}(p-r,q-s)_T$ .  Finally $<(0,0)_T,(r,s)_T>$ is zero unless $r$ and $s$ are both divisible by $N$, in which case it is $2(r,s)_T$. This makes the computation of the matrix of the pairing with respect to the basis $\mathcal{B}'$ straightforward.

Here is the matrix of the pairing for $N=3$ with respect to the basis $(p,q)_T$, where $p,q\in \{0,1,2\}$  ordered lexicographically.

\[ \begin{pmatrix} 2(0,0)_T & 0 & 0 & 0 & 0 & 0 & 0 & 0 & 0 \\
0 & (0,0)_T & (0,3)_T & 0 & 0 & 0 & 0 & 0 & 0 \\
0 & (0,3)_T & (0,0)_T & 0 & 0 & 0 & 0 & 0 & 0 \\
0 & 0 & 0 & (0,0)_T & 0 & 0 & (3,0)_T & 0 & 0\\
0 & 0 & 0 & 0 & (0,0)_T& 0 & 0 & 0 & (3,3)_T\\
0 & 0 & 0 & 0 &0 & (0,0)_T& 0 &-(3,3)_T& 0 \\
0 & 0 & 0 &(3,0)_T & 0 & 0 & (0,0)_T & 0 & 0\\
0 & 0 & 0 & 0 &0 & -(3,3)_T & 0 & (0,0)_T & 0 \\
0 & 0 & 0 & 0 &(3,3)_T & 0 & 0 & 0 & (0,0)_T \end{pmatrix}\]

The determinant of the pairing can be computed similarly to the case of the annulus.  

\begin{theorem} The determinant of the pairing $<\ , \ >$ on $S^{-1}K_N(\Sigma_{1,0})$ with respect to the basis $\mathcal{B}'$ is
\[2(0,0)_T((0,0)_T^2-(N,0)_T^2)^{\frac{N-1}{2}}((0,0)_T^2-(0,N)_T^2)^{\frac{N-1}{2}}((0,0)_T^2-(N,N)_T^2)^{\left(\frac{N-1}{2}\right)^2}.\] Hence, $S^{-1}K_N(\Sigma_{1,0})$ is a Frobenius algebra over $S^{-1}\chi(\Sigma_{1,0})$.

\end{theorem}

\proof The matrix decomposes into blocks, coming from partitioning the basis $\{(p,q)_T\}$ into subsets, where pairing basis elements in different subsets is always zero.

\begin{itemize}
\item The first subset is $\{(0,0)_T\}$.
$Tr((0,0)_T^2)=2(0,0)_T$.
\item The second subset is $\{(p,0)_T\}$ where $p\in \{1,\ldots, N-1\}$. As a matrix the pairing agrees with the matrix for the pairing of the annulus restricted to the subspace $\{T_k(x)\}$ where $k\in \{1,\ldots, N-1\}$ via the obvious correspondence.
The determinant of this block is $((0,0)_T^2-(N,0)^2)^{\frac{N-1}{2}}$.
\item The third subset is $\{(0,p)_T\}$ where $p\in \{1,\ldots,N-1\}$. The determinant of this block is $((0,0)_T^2-(0,N)_T^2)^{\frac{N-1}{2}}$.
\item There are $(\frac{N-1}{2})^2$ subsets of four that come from a choice of $p,q\in \{1,\ldots,\frac{N-1}{2}\}$ consisting of $\{(p,q)_T,(p,N-q)_T,(N-p,q)_T,(N-p,N-q)_T\}$ .  With respect to this basis the pairing is,
\[ \begin{pmatrix} (0,0)_T & 0 & 0 &(-1)^{p+q}(N,N)_T\\ 0 & (0,0)_T & (-1)^{p+q+1}(N,N)_T & 0 \\
0 & (-1)^{p+q+1}(N,N)_T & (0,0)_T & 0 \\ (-1)^{p+q}(N,N)_T & 0 & 0 &(0,0)_T\end{pmatrix} .\] The determinant of this matrix is $((0,0)_T^2
-(N,N)_T^2)$.

\end{itemize}

The determinant of the pairing is the product of the determinants of the blocks,
\[ 2(0,0_T)((0,0)_T^2-(N,0)^2)^{\frac{N-1}{2}}((0,0)_T^2-(0,N)_T^2)^{\frac{N-1}{2}}((0,0)_T^2
-(N,N)_T^2)^{\left(\frac{N-1}{2}\right)^2
}.\]

 \qed

The character variety of the torus is describe as follows. Define $t:Rep(\pi_1(SL_2\mathbb{C}))\rightarrow \mathbb{C}^3$ be given by,
\[t(\rho)=(x,y,z),\] where $x=-tr(\rho((1,0))$, $y=-tr(\rho(0,1))$ and $z=-tr(\rho((1,1))$.  The image of $t$ is in one to one correspondence with the $SL_2\mathbb{C}$-character variety of $\pi_1(\Sigma_{1,0})$.  The image of $t$ is the locus,
\[ \{(x,y,z)\in\mathbb{C}^3|x^2+y^2+z^2+xyz-4=0\}.\]

As long as we stay away from the places where $x,y$, or $z$ is $\pm 2$, the determinant of the pairing specialized at that point is nonzero.

\begin{theorem} Let $\phi_{x,y,z}:\chi(\Sigma_{1,0})\rightarrow \mathbb{C}$ be
the place corresponding to a representation where $tr(\rho((1,0)=-x$, $tr(\rho(0,1))=-y$ and $tr(\rho((1,1))=-z$.  As long as none of $x$, $y$, or $z$ takes on the values $\pm 2$, $K_N(\Sigma_{1,0})_{\phi_{(x,y,z)}}$ is a Frobenius algebra.\end{theorem}

\section{The once punctured torus }

Let $\delta$ denote the blackboard framed link based on a simple closed curve that is parallel to the boundary of $\Sigma_{1,1}$. The skein $\eta=\delta+(A^2+A^{-2})$ was introduced by Bullock and Przytycki in \cite{F}.   The inclusion of $\Sigma_{1,1}$ into $\Sigma_{1,0}$ obtained by gluing in a disk, induces a homomorphism of skein algebras.  If $(\eta)$ denotes the principal ideal of $K_N(\Sigma_{1,1})$ generated by $\eta$ then the inclusions form  a short exact sequence, 
\[\begin{CD} 0 @>>> (\eta) @>>> K_N(\Sigma_{1,1}) @>>> K_N(\Sigma_{1,0}) @>>> 0  \end{CD}.\]

The sequence splits as any two simple closed curves in $\Sigma_{1,1}$ that get mapped by inclusion to isotopic, nontrivial simple closed curves in $\Sigma_{1,0}$ are in fact isotopic in $\Sigma_{1,1}$. Hence it makes sense to talk about the skein $(p,q)_T$ in $\Sigma_{1,1}$.  The skeins $\eta^k(p,q)_T$ where $k$ ranges over the non-negative integers and $(p,q)_T$ ranges overs pairs of integers so that that $p\geq 0$ and if $p=0$, $q\geq 0$ is a basis for $\Sigma_{1,1}$ over the complex numbers.  

As a consequence of the splitting of the sequence, any skein in $K_N(\Sigma_{1,1})$ can be written as $\sum_{i,j}\alpha_{i,j}(p_i,q_i)_T+ \epsilon$ where $\epsilon \in (\eta)$. Similar to the torus case, define the weight of $\eta^k(p,q)_T$ to be $|p|+|q|$.  Define the {\bf weight} of a  skein to be the maximum weight of an $\eta^k(p,q)$ appearing with nonzero coefficient in the expansion of the skein in terms of the basis $\eta^k(p,q)$.

\begin{theorem}[Product to Sum Formula] \cite{E}  If $(p,q), (r,s) \in K_N(\Sigma_{1,1})$ then
\[(p,q)_T*(r,s)_T=A^{\left|\begin{matrix} p & q \\ r & s \end{matrix}\right|}(p+r,q+s)_T +A^{-\left|\begin{matrix} p & q \\ r & s \end{matrix}\right|}(p-r,q-s)_T +\epsilon\] where
$\epsilon\in (\eta)$ and the weight of $\epsilon$ is less than or equal to $|p|+|q|+|r|+|s|-4$. \end{theorem}

\begin{lemma}
For $\epsilon$ as in the previous theorem the highest power of $\eta$ appearing in $\epsilon$ is less than or equal to $\big\lfloor \frac{min\{|p|+|r|,|q|+|s|\}}{2} \big\rfloor$.
\end{lemma}

\begin{proof}
Visualize the punctured torus as the square with all four corners removed and this simply follows from the fact that when multiplying, you will have $|p|+|r|$ strands through the top and bottom of your square and $|q|+|s|$ through the sides. To produce an element of $(\eta)$ you need to borrow two strands from every side of your square and hence this process will be stopped when or before we reach $min\{|p|+|r|,|q|+|s|\}$. Since we need two strands from every side of our square to produce an element of $(\eta)$ then the degree must be less than or equal to $\big\lfloor \frac{min\{|p|+|r|,|q|+|s|\}}{2}\big\rfloor$.
\end{proof}

The center $Z(K_N(\Sigma_{1,1}))$ of $K_N(\Sigma_{1,1})$ is spanned by all skeins of the form $\eta^k(Na,Nb)_T$.  If $\phi:\chi(\Sigma_{1,1})\rightarrow \mathbb{C}$ is a place, supposing that $\phi(T_N(\delta))=z+z^{-1}$ then choosing any $n$th root $w$ of $z$ extends to a place $\tilde{\phi}:Z(K_N(\Sigma_{1,1}))\rightarrow \mathbb{C}$ that sends $\delta$ to $w$.  This fits in with the concept of a {\bf classical shadow} of a representation of the skein algebra due to Bonahon and Wong \cite{C}.

The skein algebra $K_N(\Sigma_{1,1})$ is the first skein algebra of a surface with boundary that is noncommutative that we have considered. We begin by treating $K_N(\Sigma_{1,1})$ as a ring extension of its center.  

\begin{proposition} The algebra $K_N(\Sigma_{1,1})$ as a module over $Z(K_N(\Sigma_{1,1}))$  is spanned by a lift of the spanning set $\mathcal{B}$ of
$K_N(\Sigma_{1,0})$ over $\chi(\Sigma_{1,0})$. \end{proposition}

\proof By the exact sequence from \cite{F} every skein can be written $\sum_{a,b,k}\alpha_{a,b,k}\eta^k(a,b)_T$ where the $\alpha_{a,b,k}\in \mathbb{C}$ and only finitely many are nonzero.   The complexity of a skein is the ordered triple consisting of its weight, the highest power of $\eta$ appearing in a term that realizes it's weight and the number of terms of highest weight with $\eta$ raised to the highest power. We order the complexity lexicographically.

Suppose that everything with complexity less than $(r,n)$ can be written as a $Z(K_N(\Sigma_{1,1}))$ linear combination of the $(a,b)$ as in the statement of the proposition, and suppose that $\sum_{k,a,b}\alpha_{k,a,b}\eta^k(a,b)_T$ has complexity $(r,n,m)$.  Choose a term $\alpha_{k,a,b}\eta^n(a,b)_T$ of weight $r$.
If $|b|>N-1$ then either we can write $b=c+N$ where $|c|<|b|$ and either $|c-N|<|c|$ or $|c|<|N|$, or we can write $b=c-N$ where $|c|<|b|$ and either $|c+N|<|c|$ or $|c+N|<N$.

Using the product to sum formula
\[ \eta^n(0,dN)_T*(a,c)_T-\eta^nA^{-dNa}(a-Nd,c-Nd)_T=A^{dNc}\eta^k(a,b)_T+\epsilon.\]
Notice the weight of $\epsilon$ less than or equal to $|a|+|c|+|dN|-4$. 

Substituting in we have reduced the complexity  by lowering the number of terms of weight $r$ with $\eta^n$.  If $b$ is between $-N+1$ and $N-1$ and $|a|>N-1$ perform the analogous rewriting to reduce the complexity.  \qed

Recall that if $S\subset R$ is a subring of the integral domain $R$ and $R$ is an integral extension of $S$ then if $S$ is a field then $R$ is a field.  The center of $K_N(\Sigma_{1,1})$ is also spanned by all $\delta^k(Np,Nq)_T$ where $p,q\in \mathbb{Z}$.  The reason is that $\delta$ is in the center, and $\eta^k$ is a polynomial in $\delta$ with lead coefficient $\delta^k$. Next,  $\chi(\Sigma_{1,1})$ is spanned by all $T_{aN}(\delta)(Np,Nq)_T$.  From this we see that $Z(K_N(\Sigma_{1,1}))$ is an integral extension of degree $N$ of $\chi(\Sigma_{1,1})$, where the added element $\delta$ satisfies a monic degree $N$ polynomial with coefficients in $\chi(\Sigma_{1,1})$.
Specificially 
\[ \delta^N=T_N(\delta)-\sum_{i=1}^{\floor{N/2}}(-1)^i\frac{N}{N-i}\binom{N-i}{i}\delta^{N-2i}\] 

This means that the ring of fractions coming from inverting all nonzero elements of $\chi(F)$ is isomorphic to the ring of fractions coming from inverting all elements of $Z(K_N(\Sigma_{1,1}))$.

\begin{theorem}  Letting $S=\chi(\Sigma_{1,1})-\{0\}$, the algebra $S^{-1}K_N(\Sigma_{1,1})$ has dimension $N^3$ over $S^{-1}\chi(\Sigma_{1,1})$, the function field of the $SL_2\mathbb{C}$-character variety of $\pi_1(\Sigma_{1,1})$. A basis is given by $T_k(\delta)(p,q)_T$ where $k,p,q\in \{0,1,\ldots,N-1\}$. \end{theorem}

\proof  The proof is in two steps.  First as a vector space over field of fractions of the center, $K_N(\Sigma_{1,1})$ has dimension $N^2$ with basis $(p,q)_T$ where $p,q\in \{0,1,\ldots, N-1\}$.  The product to sum formula for $\Sigma_{1,1}$ says that the elements of $\mathcal{B}$ satisfy relations whose leading term is the same as in $\Sigma_{1,0}$ plus lower complexity terms that are multiplied by powers of $\eta$. By continually eliminating the lead term, we terminate with an expression for every element of $\mathcal{B}'$ as a linear combination with coefficients in $S^{-1}Z(K_N(\Sigma_{1,1})$.

Since $S^{-1}Z(K_N(\Sigma_{1,1}))$ is a vector space over $S^{-1}\chi(\Sigma_{1,1})$ with basis $T_k(\delta)$ where $k\in \{0,1,\ldots,N-1\}$ we get the desired result by expanding the coefficients.  \qed

Next we need to compute the trace of the matrices corresponding to left multiplicaiton by $T_k(\delta)(p,q)_T$ where $k,p,q$ range over $\{0,1,\ldots,N-1\}$
Since any simple curve $(p/d,q/d)$ can be taken to $(1,0)$ by an orientation preserving homeomorphism ( and that homeomorphism induces an automorphism of $K_N(\Sigma_{1,1})$. We only need to compute the trace of $T_k(\delta)*(d,0)_T)$.  This is pretty easy as we can use the basis $T_i(\delta)(p,0)_T*(0,q)_T$ where $i,q$ and $q$ range from $\{0,\ldots,N-1\}$.  From the product to sum formula,

\begin{align*}
(T_k(\delta)*(d,0)_T)(T_i(\delta)(p,0)_T*(0,q)_T) &=T_{k+i}(\delta)(p+d,0)_T*(0,q)_T \\
&+T_{k+i}(\delta)(|p-d|,0)_T*(0,q)_T \\
&+T_{|k-i|}(\delta)(p+d,0)_T*(0,q)_T\\
&+T_{|k-i|}(\delta)(|p-d|,0)_T*(0,q)_T.
\end{align*}

 This  needs to be rewritten when  $p+d\geq N$ or $k+i\geq N$, but it works just in the case of the annulus.

\begin{theorem} The normalized trace of $\sum \alpha_{k,p,q}T_k(\delta)(p,q)_T$ where the $\alpha_{k,p,q}\in S^{-1}\chi(\Sigma_{1,1})$, can be computed by deleting all terms where $N$ does not divide all of $k$, $p$, and $q$. \end{theorem}  

\qed

\begin{theorem} The algebra $S^{-1}K_N(\Sigma_{1,1})$ with the normalized trace $Tr:S^{-1}K_N(\Sigma)\rightarrow S^{-1}\chi(\Sigma_{1,1})$ is Frobenius.\end{theorem}

\proof Just as in the case of $S^{-1}K_N(\Sigma_{1,0})$ except work with
basis $T_k*(\delta)(p,0)_T*(0,q)_T$.  Show that any nontrivial principal ideal contains $T_0(\delta)*(0,0)_T+\sum_{k\neq 0,p,q}\beta_{k,p,q}T_k(\delta)*(p,0)_T*(0,q)_T$ whose trace is obviously nonzero. \qed

The fundamental group of $\Sigma_{1,1}$ is the free group on two generators, it's character variety is $\mathbb{C}^3$ where the coordinates can be taken as the negative of the traces of the matrices that $(1,0)_T$, $(0,1)_T$ and $(1,1)_T$ are sent to.  We have not computed the determinant of the trace pairing, so we can't say exactly what the locus is where the algebras fail to be Frobenius is, but the determinant of the pairing is a nonzero element of the function field of the character variety and the divisor of zeroes and poles is a proper subvariety away from which $K_N(\Sigma_{1,1})_{\phi}$ is Frobenius.  The zeroes correspond to points where the pairing is degenerate and the poles correspond to points where the global  basis doesn't actually span the specialization of the algebra.

\begin{theorem} Away from a proper subvariety $C$ of the character variety, $K_N(\Sigma_{1,1})_{\phi}$ is Frobenius. \end{theorem} \qed

\bibliographystyle{amsplain}

\begin{thebibliography}{10}

\bibitem {E} Bloomquist, Wade; Frohman, Charles, {\em Multiplying in the Skein Algebra of a Punctured Torus}, Comment. Math. Helv. {\bf 78} (2003) 1-17.
\bibitem {C}  Bonahon Francis;  Wong, Helen, {\em Representations of the Kauffman skein algebra I: invariants and miraculous cancellations}, arXiv:1206.1638 [math.GT]
\bibitem {D} Bullock Doug, {\em Rings of $SL_2(\mathbb{C}$-characters and the Kauffman bracket skein module}, Comment. Math. Helv. {\bf 72} (1997), no. 4, 521542

\bibitem{BFK}  Bullock, Doug; Frohman, Charles; Kania-Bartoszynska, Joanna {\em Understanding the Kauffman bracket skein module}, J. Knot Theory Ramifications {\bf 8} (1999), no. 3, 265-277.
\bibitem {F} Bullock, Doug;   Przytycki, J\'{o}zef {\em Multiplicative Structure of Kauffman bracket Skein Module Quantizations}, Amer. Math. Soc. {\bf 128} (1999), no.3, 923-931.


\bibitem {B}  Frohman, Charles; Gelca Razvan, {\em Skein modules and the noncommutative torus}, Trans. Amer. Math. Soc. {\bf 352} (2000), no. 10, 48774888.
\bibitem{FK} Frohman, Charles; Kania-Bartoszynska, Joanna {\em The Kauffman Bracket Skein Module of $ \#_2 S^1 \times S^2$ at a  Root of Unity}, Preprint

\bibitem{HP}  Hoste, Jim; Przytycki, J\'{o}zef H. {\em The $(2,\infty)$-skein module of lens spaces; a generalization of the Jones polynomial.} J. Knot Theory Ramifications {\bf 2} (1993), no. 3, 321-333.

\bibitem{K}	Kock, Joachim {\em Frobenius algebras and 2D topological quantum field theories}, London Mathematical Society Student Texts, 59. 

Cambridge University Press, Cambridge, 2004. xiv+240 pp. ISBN: 0-521-83267-5.
\bibitem {PS} Przytycki, J\'{o}zef ;  Sikora, Adam, {\em On skein algebras and $SL_2(\mathbb{C})$-charater varieties}, Topology, {\bf 39} (2000) 115-148.


\bibitem{SW} Sikora, Adam S.; Westbury, Bruce W. {\em Confluence theory for graphs}, Algebr. Geom. Topol. {\bf 7} (2007), 439-478.










\end{thebibliography}

\end{document}